\documentclass[10pt,a4paper]{amsart}
\usepackage[utf8]{inputenc}
\usepackage[T1]{fontenc}
\usepackage[marginratio=1:1]{geometry} 
\usepackage[none]{hyphenat}
\usepackage{etoolbox}
\usepackage{lipsum}
\makeatletter
\newcommand{\addresseshere}{%
  \enddoc@text\let\enddoc@text\relax
}
\makeatother
\usepackage{rotating}
\usepackage{pdflscape}

\usepackage[left]{lineno}

\usepackage{etex}
\usepackage{listings}

\usepackage{amsmath, amsfonts, amsthm, amssymb, amsxtra, bbm, enumerate, stmaryrd}
\usepackage[UKenglish]{babel}
\usepackage[UKenglish]{isodate}
\usepackage[dvipsnames]{xcolor}

\usepackage{cite}
\usepackage{graphicx}
\usepackage[all, cmtip]{xy}
\xyoption{curve}

\usepackage{tikz}
\usetikzlibrary{arrows,calc,through,backgrounds,matrix,decorations.pathmorphing,positioning}

\usepackage{soul}
\usepackage{mathtools}

\usepackage{calligra}
\usepackage[mathcal]{euscript}
\DeclareMathAlphabet{\mathcalligra}{T1}{calligra}{m}{n}
\DeclareMathAlphabet{\mathpzc}{OT1}{pzc}{m}{it}

\usepackage{stackrel}
\usepackage{rotating}

\usepackage{wasysym}
\usepackage{paralist}
\usepackage{textcomp}
\usepackage[colorinlistoftodos]{todonotes}
\usepackage{caption}

\usepackage[bottom]{footmisc}
\usepackage{url}

\usepackage[colorlinks=false,
            linkcolor=RoyalBlue,
            anchorcolor=Periwinkle,
            citecolor=Orange,
            urlcolor=Emerald,
            bookmarks=false
            linkcolor=red,
            citecolor=blue,
]{hyperref}
\usepackage{hyperref}

\usepackage{mathrsfs} 

\calclayout
\makeatletter
\g@addto@macro{\thm@space@setup}{\thm@headfont{\bf}}

\newtheorem{lem}{Lemma}[section]
\newtheorem{prop}[lem]{Proposition}

\newtheorem{thm}[lem]{Theorem}

\theoremstyle{remark}
\newtheorem{rem}[lem]{Remark}

\theoremstyle{definition}

\newtheorem{defn}[lem]{Definition}

\newtheorem{conv}[lem]{Convention}

\numberwithin{equation}{section}

\usepackage{dsfont}

\newcommand{\fK}{\mathscr{K}}

\DeclareMathOperator{\Hom}{Hom}

\DeclareMathOperator{\op}{op}

\DeclareMathOperator{\HH}{HH}

\DeclareMathOperator{\im}{im}

\DeclareMathOperator{\one}{\mathds{1}}

\DeclareMathOperator{\Ab}{\mathbf{Ab}}

\DeclareMathOperator{\MonObj}{\mathbf{MonObj}}
\DeclareMathOperator{\lModObj}{\mathbf{leftModObj}}
\DeclareMathOperator{\rModObj}{\mathbf{rightModObj}}
\DeclareMathOperator{\bModObj}{\mathbf{biModObj}}

\DeclareMathOperator{\ot}{\otimes}

\begin{document}
\sloppy

\title{Hochschild cohomology, monoid objects and monoidal categories}
\author{Magnus Hellstrøm-Finnsen}
\address{Magnus Hellstr{\o}m-Finnsen\\ Avdeling for ingeniørfag\\ Høgskolen i Østfold\\ Postboks 700\\ NO-1757 Halden\\ Norway.}
\email{mhellstroemfinnsen@gmail.com}
\thanks{}

\maketitle

\begin{abstract}
This paper expands further on a category theoretical formulation of Hochschild cohomology for monoid objects in monoidal categories enriched over abelian groups, which has been studied in \cite{hel-18}. This topic was also presented at ISCRA, Isfahan, Iran, April 2019. The present paper aims to provide a more intuitive formulation of the Hochschild cochain complex and extend the definition to Hochschild cohomology with values in a bimodule object. In addition, an equivalent formulation of the Hochschild cochain complex in terms of a cosimplicial object in the category of abelian groups is provided. 
\end{abstract}

\tableofcontents

\section*{Introduction} %

Hochschild cohomology was initially studied by G.\ Hochschild in \cite{hoc-45} and \cite{hoc-46}, and provides a cohomology theory for associative algebras. In \cite{ger-63}, M.\ Gerstenhaber discovered that the cohomology ring has a rich structure, which later has been called a Gerstenhaber algebra. The rich structure provides interesting implications, not only restricted to mathematics, but also to physics and related fields. 

At Isfahan School \& Conference on Representations of Algebras (ISCRA), April 2019, I reported from \cite{hel-18}. This article gives a description of Hochschild cohomology in terms of monoid objects (``ring like objects'') in $\Ab$-enriched monoidal categories. Monoidal categories were independently discovered by J.\ Bénabou and S.\ Maclane in the beginning of the 1960's (see \cite{ben-63}, \cite{ben-64} and \cite{mac-63}), and they provided an axiomatic system to describe the categories with tensor product, like modules over a ring $R$ with tensor product over $R$ i.e.\ $\ot_R$, or abelian groups with the tensor product over the integers, i.e.\ $\ot_{\mathbb{Z}}$, etc.  

In this paper, we will first improve the construction of the Hochschild cochain complex given in \cite{hel-18}, by taking a more intuitive (and perhaps less combinatorial) approach to this complex. Thereafter, we approach Hochschild cohomology by a cosimplicial object in the category of abelian groups. We will discover that these formulations are equivalent. 

\section{Monoidal categories, monoid objects and module objects}%

First, we recall the definition of a monoidal category. 

\begin{defn}\label{def:nattensorcat}
A category $\fK$ is said to be {\em a monoidal category} if it is equipped with a bifunctor 
\begin{align*}
\ot : \fK \times \fK \to \fK,
\end{align*} 
called the {\em tensor product} or {\em monoidal product}, and an object $\one\in\fK$, called the {\em tensor unit} or {\em monoidal unit}, together with three natural isomorphisms: 
\begin{itemize}
\item The \emph{associator}, $\alpha:( ? \ot ? ) \ot ? \rightharpoonup ? \ot ( ? \ot ? )$, which has components:  
\begin{align*} 
\alpha_{X,Y,Z}:(X\ot Y)\ot Z \to X\ot (Y\ot Z),
\end{align*}
for all objects $X$, $Y$ and $Z$ in $\fK$.
\item The \emph{left unitor}, $\lambda: \one \ot ? \rightharpoonup ? $,  which has components:  
\begin{align*}
\lambda_X: \one \ot X \to X, 
\end{align*}
for every object $X$ in $\fK$. 
\item The \emph{right unitor}, $\rho: ? \ot \one \rightharpoonup ? $, which has components:  
\begin{align*}
\rho_X: X \ot \one \to X,
\end{align*}
for every object $X$ in $\fK$. 
\end{itemize} 
Such that the {\em pentagon diagram}:  
\begin{center}
\begin{tikzpicture}
\matrix(m)[matrix of math nodes,row sep=5em,column sep=-2em,text height=1.5ex,text depth=0.25ex]
{                                  &                                  &(W\ot X) \ot (Y \ot Z)&                                  &                                  \\
 ((W\ot X) \ot Y) \ot Z&                                  &                                  &                                  &W\ot (X \ot (Y \ot Z)),\\
                                   &(W\ot (X \ot Y)) \ot Z&                                  &W\ot ((X \ot Y) \ot Z)&                                  \\};
\draw[ ->,font=\scriptsize](m-2-1) edge         node[above,sloped]{$\alpha_{W \ot X,Y,Z}          $} (m-1-3);
\draw[ ->,font=\scriptsize](m-1-3) edge         node[above,sloped]{$\alpha_{W,X,Y \ot Z}          $} (m-2-5);

\draw[ ->,font=\scriptsize](m-2-1) edge         node[left ]{$\alpha_{W,X,Y}\ot 1_Z         $} (m-3-2);
\draw[ ->,font=\scriptsize](m-3-2) edge         node[below]{$\alpha_{W,X\ot Y,Z}           $} (m-3-4);
\draw[ ->,font=\scriptsize](m-3-4) edge         node[right]{$1_W\ot\alpha_{X,Y,Z}          $} (m-2-5);
\end{tikzpicture}
\end{center}
where $W$, $X$, $Y$ and $Z$ are arbitrary objects in $\fK$, and the \emph{triangle diagram}:
\begin{center}
\begin{tikzpicture}
\matrix(m)[matrix of math nodes,row sep=3em,column sep=2em,text height=1.5ex,text depth=0.25ex]
{ (X \ot \one) \ot Y  &             & X \ot (\one \ot Y) \\
                           & X \ot Y, &                       \\};
\draw[ ->,font=\scriptsize](m-1-1) edge         node[above]{$ \alpha_{X,I,Y}        $} (m-1-3);                       
\draw[ ->,font=\scriptsize](m-1-1) edge         node[left ]{$ \rho_X \ot 1_Y    $} (m-2-2);                       
\draw[ ->,font=\scriptsize](m-1-3) edge         node[right]{$ 1_X \ot \lambda_Y $} (m-2-2);
\end{tikzpicture} 
\end{center}
where $X$ and $Y$ are arbitrary objects in $\fK$, commute. This category is denoted $(\fK,\ot,\one,\alpha,\lambda,\rho)$. 
\end{defn}

\begin{rem}\label{rem:coh}
We recall some facts about natural coherence in monoidal categories from \cite[Section VII.2]{mac-98}. This was originally proposed by G.\ M.\ Kelly and S.\ MacLane (see \cite{mac-63} and \cite{kel-64}). Natural coherence asserts that every formal diagram involving instances (of compositions) of the natural isomorphisms ($\alpha$, $\lambda$ and $\rho$, possibly tensored with suitable identities) commutes. MacLane's coherence theorem in \cite{mac-63} can equivalently be stated as: {\em Each monoidal category is monoidally equivalent to a strict one}. A monoidal category is {\em strict} whenever the natural isomorphisms $\alpha$, $\lambda$ and $\rho$ are identities. For a strict monoidal category, the construction of the Hochschild cochain complex (in Definition~\ref{def:hoch}) is immediate. The originality of the present paper is not to rely on this theorem. How a monoidal category can be turned into a strict one is elaborated upon 
in \cite{sch-01}. 
\end{rem}

The first formulation of natural coherence above (in Remark~\ref{rem:coh}) will be important in the formulation of the Hochschild cochain complex in Section~\ref{sec:tuphoch}. In particular, it is important in the proof that this construction actually is a cochain complex (see \cite[Theorem~3.2]{hel-18}). 

The next objective is to establish the notion of ``ring like objects'' in a monoidal category, that will be utilized in this paper. Furthermore, we will also establish an appropriate notion of arrows between these objects. 

\begin{defn}\label{def:monoid} 
Let $(\fK,\ot,\one,\alpha,\lambda,\rho)$ be a monoidal category. A \emph{monoid object} is an object $M$ in $\fK$ equipped with a \emph{multiplication rule} $\mu_M:M \ot M \to M$ and a \emph{multiplicative unit} $\eta_M:\one \to M$. These morphisms satisfy the following relations: 
\begin{itemize}
\item The \emph{associative relation}: the multiplication rule is associative in the sense that the following diagram commutes:  
\begin{center}
\begin{tikzpicture}
\matrix(m)[matrix of math nodes,row sep=2.6em,column sep=2.8em,text height=1.5ex,text depth=0.25ex]
{(M \ot M)\ot M  &               & M \ot(M \ot M)  \\
  M           \ot M  &               & M \ot M             \\
                         &  M.            &                         \\};
\draw[ ->,font=\scriptsize](m-1-1) edge         node[above]{$\alpha_{M,M,M}          $} (m-1-3);
\draw[ ->,font=\scriptsize](m-1-1) edge         node[left ]{$\mu_M \ot 1_R           $} (m-2-1);
\draw[ ->,font=\scriptsize](m-2-1) edge         node[below]{$\mu_M                   $} (m-3-2);

\draw[ ->,font=\scriptsize](m-1-3) edge         node[right]{$1_M   \ot \mu_M         $} (m-2-3);
\draw[ ->,font=\scriptsize](m-2-3) edge         node[below]{$\mu_M                   $} (m-3-2);
\end{tikzpicture}
\end{center}
\item The \emph{unitarity relation}: the multiplication rule admits a left unit and a right unit in the sense that the following diagram commutes:  
\begin{center}
\begin{tikzpicture}
\matrix(m)[matrix of math nodes,row sep=2.6em,column sep=2.8em,text height=1.5ex,text depth=0.25ex]
{ M                      & \one \ot M   & M \ot M        & M \ot \one  & M     \\
                         &               & M.                  &                      \\
};
\draw[ ->,font=\scriptsize](m-1-1) edge         node[above]{$\lambda^{-1}_M          $} (m-1-2);
\draw[ ->,font=\scriptsize](m-1-2) edge         node[above]{$\eta_M\ot 1_M              $} (m-1-3);
\draw[ ->,font=\scriptsize](m-1-3) edge         node[right]{$\mu_M                   $} (m-2-3);
\draw[ ->,font=\scriptsize](m-1-1) edge         node[below]{$1_M                     $} (m-2-3);

\draw[ ->,font=\scriptsize](m-1-5) edge         node[above]{$\rho^{-1}_M             $} (m-1-4);
\draw[ ->,font=\scriptsize](m-1-4) edge         node[above]{$1_M\ot \eta_M              $} (m-1-3);
\draw[ ->,font=\scriptsize](m-1-5) edge         node[below]{$1_M                     $} (m-2-3);
\end{tikzpicture}
\end{center}
\end{itemize}
We denote a monoid object as a triple $(M,\mu_M,\eta_M)$, and often the subscripts are skipped when the monoid object is not changed. 
\end{defn} 

The next objective is to define an appropriate notion of arrows between monoid objects in a monoidal category $\fK$, and observe that monoid objects and arrows of monoid objects form a subcategory of $\fK$. 

\begin{defn}
Let $\fK$ be a monoidal category, and let $(M,\mu_M,\eta_M)$ and $(N,\mu_n,\eta_N)$ be two monoid objects in $\fK$. A {\em morphism of monoid objects} $f:(M,\mu_M,\eta_M) \to (N,\mu_N,\eta_N)$ is a morphism in $\fK$ such that multiplication is preserved, in the sense that the following diagram commutes:  
\begin{center}
\begin{tikzpicture}
\matrix(m)[matrix of math nodes,row sep=2.6em,column sep=2.8em,text height=1.5ex,text depth=0.25ex]
{
M \ot M & N \ot N \\
M       & N,      \\ 
};
\draw[ ->,font=\scriptsize](m-1-1) edge         node[above]{$ f \ot f $} (m-1-2);
\draw[ ->,font=\scriptsize](m-2-1) edge         node[above]{$ f       $} (m-2-2);
\draw[ ->,font=\scriptsize](m-1-1) edge         node[left ]{$ \mu_M    $} (m-2-1);
\draw[ ->,font=\scriptsize](m-1-2) edge         node[right]{$ \mu_N    $} (m-2-2);
\end{tikzpicture}
\end{center}
and the units are preserved, in the sense that the following diagram commutes: 
\begin{center}
\begin{tikzpicture}
\matrix(m)[matrix of math nodes,row sep=2.6em,column sep=2.8em,text height=1.5ex,text depth=0.25ex]
{
  & \one &    \\
M &      & N.  \\
};
\draw[ ->,font=\scriptsize](m-1-2) edge         node[left ]{$ \eta_M $} (m-2-1);
\draw[ ->,font=\scriptsize](m-1-2) edge         node[right]{$ \eta_N $} (m-2-3);
\draw[ ->,font=\scriptsize](m-2-1) edge         node[above]{$ f   $} (m-2-3);
\end{tikzpicture}
\end{center}
Monoid objects and morphisms of monoid objects form a category, denote this category by $\MonObj(\fK)$.  
\end{defn}

Over a ``ring-like'' object $M$ in $\fK$, there are notions of left, right and bi ``module-like'' objects in $\fK$. The next objective is to define these. 

\begin{defn}
Let $(\fK,\ot,\one,\alpha,\lambda,\rho)$ be a monoidal category and let $(M,\mu,\eta)$ be a monoid object in $\fK$. A {\em right module object over $M$} is an object $X$ in $\fK$ equipped with a {\em right action on $X$ from $M$}, which is a morphism $\omega:X \ot M \to X $ in $\fK$, that makes the following two diagrams commute: 
\begin{center}
\begin{tikzpicture}
\matrix(m)[matrix of math nodes,row sep=2.6em,column sep=2.8em,text height=1.5ex,text depth=0.25ex]
{
 (X \ot M)\ot M  &               & X \ot(M \ot M)  \\
  X       \ot M  &               & X \ot M         \\
                 &  X            &                 \\
};
\draw[ ->,font=\scriptsize](m-1-1) edge         node[above]{$\alpha_{X,M,M}          $} (m-1-3);
\draw[ ->,font=\scriptsize](m-1-1) edge         node[left ]{$\omega \ot 1_M$} (m-2-1);
\draw[ ->,font=\scriptsize](m-2-1) edge         node[below]{$\omega        $} (m-3-2);
\draw[ ->,font=\scriptsize](m-1-3) edge         node[right]{$1_X \ot \mu   $} (m-2-3);
\draw[ ->,font=\scriptsize](m-2-3) edge         node[below]{$\omega        $} (m-3-2);
\end{tikzpicture}
\end{center}
and
\begin{center}
\begin{tikzpicture}
\matrix(m)[matrix of math nodes,row sep=2.6em,column sep=2.8em,text height=1.5ex,text depth=0.25ex]
{ X \ot \one     &                & X \ot M  \\
                 &  X.            &          \\};
\draw[ ->,font=\scriptsize](m-1-1) edge         node[above]{$1_X \ot \eta $} (m-1-3);
\draw[ ->,font=\scriptsize](m-1-1) edge         node[left ]{$\rho_X    $} (m-2-2);
\draw[ ->,font=\scriptsize](m-1-3) edge         node[right]{$\omega       $} (m-2-2);
\end{tikzpicture}
\end{center}
This right module object is denoted by $(X_M,\omega)$. Dually, a {\em left module object over $M$} is an object $Y$ in $\fK$ equipped with a {\em left action from $M$ on $Y$}, $\nu:M \ot Y \to Y $, such that the dual axioms of those for a right module object are satisfied, i.e.\ such that 
\begin{center}
\begin{tikzpicture}
\matrix(m)[matrix of math nodes,row sep=2.6em,column sep=2.8em,text height=1.5ex,text depth=0.25ex]
{
 (M \ot M)\ot Y  &     & M \ot(M \ot Y) \\
  M       \ot Y  &     & M \ot Y        \\
                 &  Y  &                \\
};
\draw[ ->,font=\scriptsize](m-1-1) edge         node[above]{$\alpha_{M,M,Y}$} (m-1-3);
\draw[ ->,font=\scriptsize](m-1-1) edge         node[left ]{$\mu \ot 1_Y   $} (m-2-1);
\draw[ ->,font=\scriptsize](m-2-1) edge         node[below]{$\nu           $} (m-3-2);
\draw[ ->,font=\scriptsize](m-1-3) edge         node[right]{$1_M \ot \nu   $} (m-2-3);
\draw[ ->,font=\scriptsize](m-2-3) edge         node[below]{$\nu           $} (m-3-2);
\end{tikzpicture}
\end{center}
and
\begin{center}
\begin{tikzpicture}
\matrix(m)[matrix of math nodes,row sep=2.6em,column sep=2.8em,text height=1.5ex,text depth=0.25ex]
{ \one \ot Y     &               & M \ot Y  \\
                 &  Y            &          \\};
\draw[ ->,font=\scriptsize](m-1-1) edge         node[above]{$ \eta \ot 1_Y $} (m-1-3);
\draw[ ->,font=\scriptsize](m-1-1) edge         node[left ]{$\lambda_Y  $} (m-2-2);
\draw[ ->,font=\scriptsize](m-1-3) edge         node[right]{$\nu$} (m-2-2);
\end{tikzpicture}
\end{center}
commute. This left module object is denoted by $({_MY},\nu)$. 
\end{defn}

\begin{defn}\label{def:modmor}
Let $(X,\omega_X)$ and $(Y,\omega_Y)$ be two right module objects over the same monoid object $M$ in a monoidal category $\fK$. A {\em morphism of right module objects over $M$}, $f: (X,\omega_X) \to (Y,\omega_Y)$, is a morphism $f:X \to Y$ in $\fK$ preserving the right action, i.e.\ the morphism respects the right module object structure in the sense that the following diagram commutes:  
\begin{center}
\begin{tikzpicture}
\matrix(m)[matrix of math nodes,row sep=2.6em,column sep=2.8em,text height=1.5ex,text depth=0.25ex]
{
 X \ot M  &               & Y \ot M     \\
 X        &               & Y.           \\
};
\draw[ ->,font=\scriptsize](m-1-1) edge         node[above]{$f \ot 1_M $} (m-1-3);
\draw[ ->,font=\scriptsize](m-1-3) edge         node[right]{$ \omega_Y $} (m-2-3);
\draw[ ->,font=\scriptsize](m-1-1) edge         node[left ]{$ \omega_X $} (m-2-1);
\draw[ ->,font=\scriptsize](m-2-1) edge         node[above]{$f         $} (m-2-3);
\end{tikzpicture}
\end{center}
Right module objects over $M$ form a category. We denote this category by $\mathbf{rightModObj}_{\fK}(M)$, or simply $\mathbf{rightModObj}(M)$. 

A {\em morphism of left module objects over $M$}, $f:(X,\nu_X)\to(Y,\nu_Y)$, is defined similarly as a morphism of right module objects over $M$, that is, $f:X\to Y$ is a morphism in $\fK$, that preserves the left action, in the sense that the following diagram commutes:  
\begin{center}
\begin{tikzpicture}
\matrix(m)[matrix of math nodes,row sep=2.6em,column sep=2.8em,text height=1.5ex,text depth=0.25ex]
{
 M \ot X  &               & M \ot Y     \\
 X        &               & Y.           \\
};
\draw[ ->,font=\scriptsize](m-1-1) edge         node[above]{$1_M \ot f $} (m-1-3);
\draw[ ->,font=\scriptsize](m-1-3) edge         node[right]{$ \nu_Y    $} (m-2-3);
\draw[ ->,font=\scriptsize](m-1-1) edge         node[left ]{$ \nu_X    $} (m-2-1);
\draw[ ->,font=\scriptsize](m-2-1) edge         node[above]{$f         $} (m-2-3);
\end{tikzpicture}
\end{center}
Left module objects over $M$ form a category. We denote this category by $\mathbf{leftModObj}_{\fK}(M)$, or simply $\mathbf{leftModObj}(M)$. 
\end{defn}

%
\begin{defn}\label{def:bimod}
Let $(\fK,\ot,\one,\alpha,\lambda,\rho)$ be a monoidal category, and let $(M,\mu_M,\eta_M)$ and $(N,\mu_N,\eta_N)$ be two monoid objects in $\fK$. An {\em $N$-$M$-bimodule object} is an object $X$ in $\fK$ which is:  
\begin{itemize}
\item A right module object, say $(X,\omega)$, over $M$. 
\item A left module object, say $(X,\nu)$, over $N$. 
\end{itemize}
Such that (in addition to the diagrams in Definition~\ref{def:modmor}) the following diagram commutes:  
\begin{center}
\begin{tikzpicture}
\matrix(m)[matrix of math nodes,row sep=2.6em,column sep=2.8em,text height=1.5ex,text depth=0.25ex]
{
 (N \ot X)\ot M  &     & N \ot(X \ot M) \\
  X       \ot M  &     & N \ot X        \\
                 &  X. &                \\
};
\draw[ ->,font=\scriptsize](m-1-1) edge         node[above]{$\alpha_{N,X,M}$} (m-1-3);
\draw[ ->,font=\scriptsize](m-1-1) edge         node[left ]{$\nu \ot 1_M   $} (m-2-1);
\draw[ ->,font=\scriptsize](m-2-1) edge         node[below]{$\omega        $} (m-3-2);
\draw[ ->,font=\scriptsize](m-1-3) edge         node[right]{$1_N \ot \omega   $} (m-2-3);
\draw[ ->,font=\scriptsize](m-2-3) edge         node[below]{$\nu           $} (m-3-2);
\end{tikzpicture}
\end{center}
We often denote bimodule objects by a triple: 
\begin{align*} 
({_NX_M},\nu^N_X,\omega^M_X)=(X,\nu,\omega).
\end{align*} 
When $N=M$, we simply say that $(X,\nu,\omega)$ is an {\em $M$-bimodule object}. 
\end{defn}

In the classical case of $S$-$R$-bimodules over two rings $R$ and $S$, a morphism of $S$-$R$-bimodules is simply a morphism $S$-$R$-modules. Next, we observe that a similar result holds for bimodule objects in a monoidal category, i.e.\ a morphism of left and right module objects respects the bimodule object structure. 
%

\begin{prop}\label{prop:bimor}
Let $(\fK,\ot,\one,\alpha,\lambda,\rho)$ be a monoidal category, let $(M,\mu_M,\eta_M)$ and $(N,\mu_N,\eta_N)$ be two monoid objects in $\fK$, and let $({_NX_M},\nu^N_X,\omega^M_X)=(X,\nu_X,\omega_X)$ and $({_NY_M},\nu^N_Y,\omega^M_Y)=(Y,\nu_Y,\omega_Y)$ be $N$-$M$-bimodule objects. A morphism of $N$-$M$-module objects  
\begin{align*}
f:(X,\nu_X,\omega_X) \to (Y,\nu_Y,\omega_Y) 
\end{align*}
(i.e.\ $f$ respects the left action and the right action) does also respect the bimodule object structure. 
\end{prop}

\begin{proof}
Consider the following diagram:   
\begin{center}
\begin{tikzpicture}
\matrix(m)[matrix of math nodes,row sep=2.6em,column sep=2.8em,text height=1.5ex,text depth=0.25ex]
{
 (N \ot X)\ot M  &                && (N \ot Y)\ot M &                \\
                 & N \ot(X \ot M) &&                & N \ot(Y \ot M) \\
  X       \ot M  &                &&  Y       \ot M &                \\
                 & N \ot X        &&                & N \ot Y        \\
  X              &                &&  Y             &                \\
                 & X              &&                &       Y.       \\
};
\draw[ ->,font=\scriptsize](m-1-1) edge[sloped]     node[above]{$\alpha_{N,X,M}   $} (m-2-2);
\draw[ ->,font=\scriptsize](m-1-1) edge[near end]   node[left ]{$\nu_X \ot 1_M    $} (m-3-1);
\draw[ ->,font=\scriptsize](m-3-1) edge[near end]   node[left ]{$\omega_X         $} (m-5-1);
\draw[ ->,font=\scriptsize](m-2-2) edge[near end]   node[left ]{$1_N \ot \omega_X $} (m-4-2);
\path[ = ,font=\scriptsize](m-5-1) edge[-,double]   node[above]{$                 $} (m-6-2);
\draw[ ->,font=\scriptsize](m-1-4) edge[sloped]     node[above]{$\alpha_{N,Y,M}   $} (m-2-5);
\draw[ ->,font=\scriptsize](m-1-4) edge[near end]   node[left ]{$\nu_Y \ot 1_M    $} (m-3-4);
\draw[ ->,font=\scriptsize](m-3-4) edge[near end]   node[left ]{$\omega_Y         $} (m-5-4);
\draw[ ->,font=\scriptsize](m-2-5) edge[near end]   node[left ]{$1_N \ot \omega_Y $} (m-4-5);
\draw[ ->,font=\scriptsize](m-4-5) edge[near end]   node[left ]{$\nu_Y            $} (m-6-5);
\path[->, font=\scriptsize](m-5-4) edge[-,double]   node[above]{$                 $} (m-6-5);
\draw[ ->,font=\scriptsize](m-1-1) edge[near start] node[above]{$(1_N  \ot f)\ot 1_M $} (m-1-4);
\draw[ ->,font=\scriptsize](m-5-1) edge[near end]   node[above]{$f$} (m-5-4);
\draw[ ->,font=\scriptsize](m-6-2) edge[near end]   node[above]{$f$} (m-6-5);
\draw[ - ,line width=2.5mm, white](m-4-2) edge[near end]   node[left ]{$            $} (m-6-2);
\draw[ ->,font=\scriptsize](m-4-2) edge[near end]   node[left ]{$\nu_X            $} (m-6-2);
\draw[ - ,line width=2.5mm, white](m-2-2) edge[near start] node[above]{$ $} (m-2-5);
\draw[ ->,font=\scriptsize](m-2-2) edge[near start] node[above]{$ 1_N  \ot(f \ot 1_M)$} (m-2-5);
\end{tikzpicture}
\end{center}
The top square commutes since $\alpha$ is a natural transformation. The back part of the diagram (the back ``hexagon'') commutes since $f$ is a morphism of $N$-$M$ module objects. The front part/hexagon commutes by the same reason. The left side of the diagram commutes since $X$ is an $N$-$M$-bimodule. Similarly, the right part commutes since $Y$ is an $N$-$M$-bimodule. The bottom square commutes by composition with identities. Hence, the diagram commutes, which proves the proposition. 
%
%
\end{proof}

A consequence of Proposition~\ref{prop:bimor} is that no further assumptions on the morphisms of bimodule objects are needed, than those already given by the module objects. The definition of the category of bimodule objects follows next. 

\begin{defn}
Let $(\fK,\ot,\one,\alpha,\lambda,\rho)$ be a monoidal category and let $(M,\mu_M,\eta_M)$ be a monoid object in $\fK$. The full subcategory of $\lModObj(M)$ and $\rModObj(M)$ generated by bimodule objects over $M$ is called the {\em category of bimodule objects over $M$}, and it is denoted by $\bModObj(M)$.   
\end{defn}

\section{Tuples of a monoid object, the Hochschild cochain complex and Hochschild cohomology}\label{sec:tuphoch} 

In this section our monoidal category $\fK$ will be $\Ab$-enriched, which we define next. 

\begin{defn}
An arbitrary category $\fK$ is said to be {\em $\Ab$-enriched} if the hom-sets are abelian groups and the composition is bilinear over the integers. 
\end{defn}

This means that $\fK$ is enriched over the category of abelian groups. However, it is not assumed that $\fK$ is additive, since we do not assume that our category has finite biproducts.  

We recall the basic definitions of a cochain complex and cohomology. 

\begin{defn}
Let $\fK$ be an $\Ab$-enriched category (or an other category with a zero object). A ($\mathbb{Z}$-{\em graded}) {\em cochain complex} in $\fK$ is a sequence of objects and arrows 
\begin{align*}
(C^{\bullet},d^{\bullet}):\dots \xrightarrow[]{d^{-2}} C^{-1} \xrightarrow[]{d^{-1}} C^0 \xrightarrow[]{d^{0}} C^1 \xrightarrow[]{d^{1}} \cdots, 
\end{align*} 
such that two adjacent arrows compose to zero, $d^{k} \circ d^{k-1} = 0$ for all $k \in \mathbb{Z}$. A {\em morphism of chain complexes} 
\begin{align*}
f: (A^{\bullet},d_A^{\bullet}) \to (B^{\bullet},d_B^{\bullet})
\end{align*}
is a degree wise collection of morphisms in $\fK$, $f^k:A^k \to B^k$ for all $k\in\mathbb{Z}$, such that all squares in the following diagram commute:  
\begin{center}
\begin{tikzpicture}
\matrix(m)[matrix of math nodes,row sep=1.5em,column sep=2.2em,text height=1.5ex,text depth=0.25ex]
{
    & \cdots & A^{k-1} & A^{k} & A^{k+1} & \cdots  \\
    & \cdots & B^{k-1} & B^{k} & B^{k+1} & \cdots. \\ 
};
\draw[ ->,font=\scriptsize](m-1-2) edge         node[above]{$          $}(m-1-3);
\draw[ ->,font=\scriptsize](m-1-3) edge         node[above]{$d^{k-1}_A $}(m-1-4);
\draw[ ->,font=\scriptsize](m-1-4) edge         node[above]{$d^k_A     $}(m-1-5);
\draw[ ->,font=\scriptsize](m-1-5) edge         node[above]{$          $}(m-1-6);
\draw[ ->,font=\scriptsize](m-2-2) edge         node[above]{$          $}(m-2-3);
\draw[ ->,font=\scriptsize](m-2-3) edge         node[above]{$d^{k-1}_B $}(m-2-4);
\draw[ ->,font=\scriptsize](m-2-4) edge         node[above]{$d^k_B     $}(m-2-5);
\draw[ ->,font=\scriptsize](m-2-5) edge         node[above]{$          $}(m-2-6);
\draw[ ->,font=\scriptsize](m-1-3) edge         node[right]{$f^{k-1}   $}(m-2-3);
\draw[ ->,font=\scriptsize](m-1-4) edge         node[right]{$f^{k}     $}(m-2-4);
\draw[ ->,font=\scriptsize](m-1-5) edge         node[right]{$f^{k+1}    $}(m-2-5);
%
\end{tikzpicture}
\end{center}
The category of cochain complexes over $\fK$ is denoted by $\mathbf{coCh}(\fK)$. 
\end{defn} 

Note, with this definition, a chain complex is a ``cochain complex'' where the $\mathbb{Z}$-grading is reversed, or equivalently, a cochain complex in $\fK^{\op}$. 

\begin{defn}
Let $\fK$ be an abelian category. The {\em $k$-th cohomology group} of the cochain complex $(C^{\bullet},d^{\bullet})$ is defined to be 
\begin{align*}
\ker(d^{k}) / \im(d^{k-1}). 
\end{align*}
\end{defn}


The next objective is to define the Hochschild cochain complex for a monoid object $M$ in an $\Ab$-enriched monoidal category with values in a bimodule object $X$ over $M$. As stated in the Introduction, we will use a slightly different 
method, than that used in the construction of this cochain complex given in \cite{hel-18}. In particular, we will deal with the associators slightly differently, and perhaps more intuitively, in order to get a less complicated description of the differentials. 
The method of construction of the cochain complex in the present paper is more motivated by the basic ideas behind Hochschild cohomology, while the operations used in \cite{hel-18} were more ``tensor combinatorially'' motivated. Hence, we will now discuss these ideas carefully to find a more intuitive formulation, but first we agree on some conventions. 

\begin{conv}\label{con:arrangement}
Let $M$ be a monoid object in a monoidal category $\fK$. We denote the $k$-tuple tensor product by 
\begin{align*}
M^{\ot k} = (\cdots (( M \ot M ) \ot M) \cdots ) \ot M, 
\end{align*}
where $M$ occurs $k$ times and the parenthesis are arranged to the left side, i.e.\ all the left parenthesis are grouped together. 
\end{conv}

Recall the basic idea of the formulation of the Hochschild cochain complex given in \cite[Definition~2.1]{hel-18}. The objects in the cochain complex here are given by 
\begin{align*}
C^k= 
\begin{cases}
	0                              &\text{for } k<0      \\
	\Hom_{\fK}(\one     ,X)        &\text{for } k=0      \\ 
	\Hom_{\fK}(M^{\ot k},X)        &\text{for } k\geq 1, 
\end{cases}
\end{align*}
for a monoidal category $\fK$, a monoid object $M$ in $\fK$ and a bimodule object $X$ over $M$. 

The differential, $d^k$, provides a morphism $d^{k}(f):M^{\ot(k+1)}\to X$, for every morphism $f:M^{\ot k} \to X$ in $\fK$. The differential is given by an alternating sum. In this sum, we will later distinguish between what is referred to as inner and outer summands.  

For the inner summands, we will apply the multiplication rule $\mu$ on a pair of adjacent objects in the tuple $M^{\ot(k+1)}$, before we apply $f$ on the ``remaining'' $M^{\ot k}$. Therefore, we need a procedure to isolate a pair of objects from the tuple $M^{\ot k}$ and a procedure to ``rearrange'' it back (when we have one isolated object). 

There are two outer summands. For the first one, we will isolate and apply $f$ on the last $k$ occurrences of $M$ in the tuple $M^{\ot(k+1)}$, then we apply the left action from the remaining monoid object $M$ on the bimodule object $X$. The second outer summand is dual to the first, we apply $f$ to the first $k$ occurrences of $M$ in the tuple $M^{\ot(k+1)}$, and then we apply the right action on $X$ from $M$. For the first of the outer summands, we need a procedure to isolate a single object at the beginning of a tuple. Such prosedure is not needed for the second outer summand, since a singe object is already isolated on the right in $M^{\ot (k+1)}$ (see Convention~\ref{con:arrangement}). 

As mentioned, a rather combinatorial, but nevertheless a general procedure to isolate a tuple with $i$ objects in a tuple with $k$ objects in position $j$ by the operation $\alpha^{i,j}_k$ is described in \cite[2.4]{hel-18}. In this paper, we describe a different procedure to isolate an adjacent pair of objects, which will be useful for the inner summands. First, we observe the following 
\begin{align*}
M^{\ot k}&=M^{\ot(k-1)}\ot M = (M^{\ot(k-2)}\ot M)\ot M = ((M^{\ot (k-3)}\ot M)\ot M)\ot M = \cdots = \\
&= (\cdots(M^{\ot 2}\ot M ) \ot M \cdots ) \ot M,
\end{align*} 
and we can apply $\mu$ to the isolated pair in the front of this tuple by applying 
\begin{align*}
M^{\ot k}=(\cdots(M^{\ot 2}\ot M ) \ot M \cdots ) \ot M 
\xrightarrow[]{%
(\cdots (\mu \ot 1_M ) \ot 1_M \cdots ) \ot 1_M } 
(\cdots(M \ot M ) \ot M \cdots ) \ot M = M^{\ot (k-1)}.
\end{align*}

To isolate a pair of adjacent objects at other places in the tuple, than the pair already isolated in the front, we use the associative relation as follows: 
\begin{align*}
M^{\ot k} 
\xrightarrow[]{%
(\cdots(\alpha_{M^{\ot i},M,M}\ot 1_M ) \ot 1_M \cdots )\ot 1_M}
(\cdots(M^{\ot i}\ot M^{\ot 2})\ot M \cdots )\ot M. 
\end{align*}
We apply $\mu$ to the isolated pair on the right side of the expression above 
\begin{align*}
(\cdots(M^{\ot i}\ot M^{\ot 2})\ot M \cdots )\ot M 
\xrightarrow[]{(\cdots ( 1_{M^{\ot i}} \ot \mu ) \ot 1_M \cdots ) \ot 1_M}
(\cdots(M^{\ot i}\ot M)\ot M \cdots )\ot M=M^{\ot (k-1)}.
\end{align*}
This procedure (in contrast to that in \cite{hel-18}) has the advantage of that we get $M^{\ot (k-1)}$ directly on the right hand side, and there is no need to ``rearrange'' the parentheses back again after applying $\mu$. 

For the first outer summand, we need to isolate one object in the beginning of a tuple. We do this by a composition of associators (tensored with the unit) of the form 
\begin{align*}
M^{\ot k} 
\xrightarrow[]{%
(\cdots(\alpha_{M,M,M} \ot 1_M )\ot 1_M \cdots )\ot 1_M}
(\cdots ( M \ot M^{\ot 2} ) \ot M \cdots )\ot M 
\xrightarrow[]{%
(\cdots ( \alpha_{M,M^{\ot 2},M})\ot 1_M )\ot 1_M \cdots )\ot 1_M} 
\\ 
(\cdots ( M \ot M^{\ot 3} ) \ot M \cdots )\ot M 
\xrightarrow[]{%
(\cdots ( \alpha_{M,M^{\ot 3},M})\ot 1_M )\ot 1_M \cdots )\ot 1_M} \cdots 
\\ 
\xrightarrow[]{%
\alpha_{M,M^{\ot (k-2)},M}} 
M \ot M^{\ot(k-1)}. 
\end{align*}
We denote this composition by 
\begin{align*}
\alpha^1_{M^{\ot k}} := \alpha_{M,M^{\ot (k-2)},M} \circ \cdots \circ (\cdots ( \alpha_{M,M^{\ot 2},M})\ot 1_M )\ot 1_M \cdots )\ot 1_M \circ (\cdots(\alpha_{M,M,M} \ot 1_M )\ot 1_M \cdots )\ot 1_M : \\ 
M^{\ot k} \to M \ot M^{\ot (k-1)}. 
\end{align*}

It should be noted that any procedure to get a specific reconfiguration of a tuple is equivalent, since natural coherence in monoidal categories (Remark~\ref{rem:coh}) provides that any diagram constructed to determine the relationship between such procedures commutes.  

Now, we can define the differential in the Hochschild cochain complex. 

\begin{defn}\label{def:hoch}
Let $(\fK,\ot,\one,\alpha,\lambda,\rho)$ be an $\Ab$-enriched monoidal category, let $(M,\mu,\eta)$ be a monoid object in $\fK$ and let $(X,\nu,\omega)$ be a bimodule object over $M$. We define the {\em Hochschild cochain complex} $(C^{\bullet},d^{\bullet})$ to have objects $C^k(M;X)$ as given above, i.e.\ 
\begin{align*}
C^k(M;X)=
\begin{cases}
0                              &\text{for } k<0      \\
\Hom_{\fK}(\one     ,X)        &\text{for } k=0      \\ 
\Hom_{\fK}(M^{\ot k},X)        &\text{for } k>0, 
\end{cases}
\end{align*}
and the differentials, 
\begin{align*}
d^k:C^{k}(M;X)\to C^{k+1}(M;X),
\end{align*}
are defined as:  
\begin{itemize}
\item For $k<0$: $d^k=0$. 
\item For $k=0$: Let $f\in C^0(M;X) = \Hom_{\fK}(\one,X)$. The differential, $d^0:\Hom_{\fK}(\one,X)\to \Hom_{\fK}(M,X)$, evaluated on $f$ is defined to be the sum of the compositions of:  
\begin{align*}
&M \xrightarrow[]{\lambda^{-1}_M} \one \ot M \xrightarrow[]{1_M \ot f} X \ot M \xrightarrow[]{\omega} X \\
&-M \xrightarrow[]{\rho^{-1}_M} M \ot \one \xrightarrow[]{f \ot 1_M} M \ot X \xrightarrow[]{\nu} X. 
\end{align*}
\item For $k>0$: We distinguish between the inner and outer summands, as discussed above. Denote the summands by $\chi_i$. Let $f\in C^k(M;X)=\Hom_{\fK}(M^{\ot k},X)$. The differential, $d^k:\Hom_{\fK}(M^{\ot k},X) \to \Hom_{\fK}(M^{\ot (k+1)},X)$, evaluated on $f$ is defined to be the alternating sum, $\sum_{i=0}^{k}(-1)^i\chi_i$, where the $\chi_i$'s are defined as follows: \\
For $i=0,$ we get the first outer summand $\chi_0$, which is defined to be the composition of 
\begin{align*}
M^{\ot (k+1)} \xrightarrow[]{\alpha^1_{M^{\ot (k+1)}}} M \ot M^{\ot k} \xrightarrow[]{1_M \ot f} M \ot X \xrightarrow[]{\nu} X. 
\end{align*} 
For $i=k$, we get the other outer summand $\chi_k$, which is defined to be the composition of 
\begin{align*}
M^{\ot (k+1)} = M^{\ot k} \ot M \xrightarrow[]{f \ot 1_M} X \ot M \xrightarrow[]{\omega} X. 
\end{align*}
For $0<i<k$, we get the inner summands $\chi_i$, which are defined to be the compositions of the form 
\begin{align*}
M^{\ot (k+1)} \xrightarrow[]{(\cdots ( \alpha_{M^{\ot i},M,M} \ot 1_M ) \ot 1_M \cdots ) \ot 1_M} (\cdots(M^{\ot i}\ot M^{\ot 2})\ot M \cdots )\ot M &\xrightarrow[]{(\cdots ( 1_{M^{\ot i}} \ot \mu ) \ot 1_M \cdots ) \ot 1_M}\\
&M^{\ot k} \xrightarrow[]{f} X. 
\end{align*}
\end{itemize}
\end{defn} 

The formulation of the Hochschild cochain complex 
is a coherently equivalent description to that in \cite[Definition~3.1]{hel-18}. Hence, \cite[Theorem~3.2]{hel-18} confirms that $(C^{\bullet}(M;X),d^{\bullet})$ is, in fact, a cochain complex, and its cohomology groups are well-defined.  

\begin{defn}\label{def:HH}
The {\em Hochschild cohomology groups} are defined to be the cohomology of the cochain complex $(C^{\bullet}(M;X),d^{\bullet})$ given in Definition~\ref{def:hoch}, i.e.\ 
\begin{align*}
\HH^{i}(M;X)=\ker(d^i)/\im(d^{i-1}). 
\end{align*}
\end{defn} 

Some of the classical results for the lower dimensional Hochschild cohomology groups are also proved in \cite{hel-18}. Here, $\HH^{0}(M;M)=\HH^{0}(M)$ is used to define some notion of ``quasi centre'' of $M$. Similarly, we can define the {\em quasi centre}, $Z(X)$, for a bimodule object $X$ to be:  
\begin{align*}
Z(X):=\HH^{0}(M;X)\cong\ker d^0. 
\end{align*}
%

With this formulation of Hochschild cohomology, it is proved that the Hochschild cohomology ring 
\begin{align*}
\HH^*(M)= \bigoplus_{i=0}^{\infty}\HH^{i}(M)
\end{align*}
(where $\HH^i(M):=\HH^i(M;M)$) is graded commutative with the cup product (see \cite[Theorem~5.5]{hel-18}). 

\section{Cosimplicial description of the Hochschild cochain complex}\label{sec:cosimphoch} 


The objective for this section is to formulate the Hochschild cochain complex of a monoid object $M$ in a monoidal category $\fK$ with values in a bimodule $X$ in terms of a cosimplicial object in $\Ab$. We argue that this formulation is equivalent to that which was given in the previous section (Section~\ref{sec:tuphoch}). In \cite[Section~9.1]{wei-88}, Hochschild cohomology for a $k$-algebra is introduced as a cosimplicial object in the vector space over $k$ ($k$ is a field).

\begin{defn}\label{def:delta}
Let $\mathbf{\Delta}$ denote category of non-empty finite ordinals and order preserving maps. The non-empty finite ordinals are denoted by 
\begin{align*}
[k]=\{0,1,\dots,k\} 
\end{align*}
with the ordering $0\leq1\leq\cdots\leq k$. This category is called the {\em simplex category} or the {\em $\mathbf{\Delta}$-category}. Let $\fK$ be an arbitrary category. A {\em cosimplicial object} $A$ in $\fK$ is a functor $A:\mathbf{\Delta} \to \fK$. We write $A([n])=A^n$.  
\end{defn}

We should first note that $A$, in the definition above, is often referred to as a {\em $\fK$-valued cosimplicial object}. It should also be noted that a {\em $\fK$-valued simplicial object} $B$ is a functor $B:\mathbf{\Delta}^{\op}\to\fK$. 

The category $\mathbf{\Delta}$ can be described several ways. Since we try to explore category theoretical properties of the Hochschild cochain complex, we should perhaps describe the $\mathbf{\Delta}$-category categorically. 

\begin{rem}\label{rem:delta}
The $\mathbf{\Delta}$-category can equivalently be defined as the category, where the objects are free categories on linear single arrowed and directed graphs/quivers and the morphisms are functors. 
\end{rem}
%
%

We recall the face and degeneracy maps in $\mathbf{\Delta}$ (see \cite[Section~8.1]{wei-88}). 

\begin{defn}\label{def:injsur}
For each $[k]\in\mathbf{\Delta}$, let $\epsilon_i:[k-1]\to[k]$ denote the ($i$-th) {\em face map}, that is the unique injective map that misses $i$: 
\begin{align*}
\epsilon_i(j)=
\begin{cases}
j   &\qquad \text{if $j<i$}\\
j+1 &\qquad \text{if $j\geq i$}.
\end{cases}
\end{align*}
Dually, for each $[k]\in\mathbf{\Delta}$, let $\zeta_i:[k+1]\to[k]$ denote the ($i$-th) {\em degeneracy map}, that is the unique surjective map that sends two elements to $i$:  
\begin{align*}
\zeta_i(j)=
\begin{cases}
j   &\qquad \text{if $j\leq i$}\\
j-1 &\qquad \text{if $j> i$}.
\end{cases}
\end{align*} 
\end{defn}
%
%

Face and degeneracy maps satisfies certain identities (see \cite[Exercise~8.1.1]{wei-88}). Let $A$ be a cosimplicial object in an arbitrary category $\fK$, then the face maps and degeneracy maps generalise to coface operations and codegeneracy operations on $A$, respectively. We recall \cite[Corollary~8.1.4]{wei-88}, which states how we can describe cosimplicial objects. 

\begin{prop}\label{prop:cosimp}
To describe a cosimplicial object $A$ in an arbitrary category $\fK$, it is sufficient 
%
and necessary to give a sequence of objects $A^0,A^1,\dots$ together with coface operations 
\begin{align*}
\delta^{i}: A^{k-1} \to A^{k} 
\end{align*}
and codegeneracy operations 
\begin{align*}
\sigma^{i}: A^{k+1} \to A^{k}, 
\end{align*}
such that the following identities are satisfied:  
\begin{align*}
\delta^j\delta^i&=\delta^i\delta^{j-1} \qquad\text{if $i<j$}, \\
\sigma^j\sigma^i&=\sigma^i\sigma^{j+1} \qquad\text{if $i\leq j$}, \\
\sigma^j\delta^i&=
\begin{cases}
\delta^i\sigma^{j-1} &\text{if $i<j$} \\
1                    &\text{if $i=j$ or $i=j+1$} \\
\sigma^{i-1}\delta^j &\text{if $i>j+1 $}. 
\end{cases}
\end{align*}\end{prop}

\begin{defn}\label{def:cosid}
We will refer to the identities given in Proposition~\ref{prop:cosimp} as the {\em cosimplicial identities}. 
\end{defn}

The objective now is to construct a cosimplicial object in $\Ab$, which we later will reformulate to the Hochschild cochain complex (in Theorem~\ref{def:comp1}).   

\begin{defn}\label{def:coshoch}
Let $(\fK,\ot,\one,\alpha,\lambda,\rho)$ be an $\Ab$-enriched monoidal category, $(M,\mu,\eta)$ a monoid object in $\fK$ and $(X,\nu,\omega)$ a bimodule object over $M$. The object $A$ in $\Ab$ is constructed in the following way. To each $[k]\in\mathbf{\Delta}$ we assign:  
\begin{align*}
[k] \mapsto 
\begin{cases}
A^0=\Hom_{\fK}(\one,X) &\text{for $k=0$}\\
A^k=\Hom_{\fK}(M^{\ot k},X) &\text{for $k>0$}. 
\end{cases}
\end{align*}
The coface operations are given by: For $k=0$, let $f\in A^0 = \Hom_{\fK}(\one,X)$. Then $\delta^i:A^0\to A^1$, for $i\in\{0,1\}$, evaluated on $f$, is defined as the composition of:   
\begin{align*}
\begin{cases}
\delta^0(f) = M \xrightarrow[]{\lambda^{-1}_M} \one \ot M \xrightarrow[]{ f \ot 1_M } X \ot M \xrightarrow[]{\omega} X \\
\delta^1(f) = M \xrightarrow[]{\rho^{-1}_M} M \ot \one \xrightarrow[]{ 1_M \ot f } M \ot X \xrightarrow[]{\nu} X.  
\end{cases}
\end{align*} 
For $k>0$, let $f\in A^k = \Hom_{\fK}(M^{\ot k},X)$. Then $\delta^i:A^k\to A^{k+1}$, for $i\in\{0,1,\dots,k\}$, evaluated on $f$, is defined as the composition of: 
\begin{align*}
\begin{cases}
\delta^0(f) = M^{\ot (k+1)} \xrightarrow[]{\alpha^1_{M^{\ot (k+1)}}} M \ot M^{\ot k} \xrightarrow[]{1_M \ot f} M \ot X \xrightarrow[]{\nu} X \\\\ 
\delta^i(f) = 
\begin{cases}
&M^{\ot (k+1)} \xrightarrow[]{(\cdots (\alpha_{M^{\ot i},M,M} \ot 1_M) \ot 1_M \cdots ) \ot 1_M} (\cdots(M^{\ot i}\ot M^{\ot 2})\ot M \cdots )\ot M \\
&\xrightarrow[]{(\cdots ( 1_{M^{\ot i}} \ot \mu ) \ot 1_M \cdots ) \ot 1_M}
M^{\ot k} \xrightarrow[]{f} X 
\end{cases}
\qquad\text{for $0<i<k$}\\\\
\delta^k(f) = M^{\ot (k+1)} = M^{\ot k} \ot M \xrightarrow[]{f \ot 1_M} X \ot M \xrightarrow[]{\omega} X.   
\end{cases}
\end{align*}
The codegeneracy operations are given by: For $k=0$, let $f\in A^1=\Hom_{\fK}(M,X)$. Then $\sigma^0: A^1 \to A^0$, evaluated on $f$, is defined as the composition of: 
\begin{align*}
\sigma^0(f) = \one \xrightarrow[]{\eta} M \xrightarrow[]{f} X. 
\end{align*}
For $k>0$, let $f\in A^{k+1}=\Hom_{\fK}(M^{\ot(k+1)},X)$. Then $\sigma^i:A^{k+1}\to A^k$, for $i\in \{0,1,\dots,k\}$, evaluated on $f$, is defined as the composition of:  
\begin{align*}
\begin{cases}
\sigma^0(f)= 
\begin{cases}
M^{\ot k} \xrightarrow[]{(\cdots((\lambda^{-1}_{M}\ot 1_M) \ot 1_M ) \cdots )\ot 1_M} (\cdots ((\one \ot M)\ot M )\ot M \cdots )\ot M \\
\xrightarrow[]{(\cdots(((\eta \ot 1_M)\ot 1_M )\ot 1_M ) \cdots )\ot 1_M} M ^{\ot (k+1)} \xrightarrow[]{f} X  \end{cases}\\\\ 
\sigma^i(f)=
\begin{cases}
M^{\ot k} 
\xrightarrow[]{%
(\cdots(1_{M^{\ot i}} \ot \lambda^{-1}_M)\ot1_M \cdots )\ot 1_M}
(\cdots(M^{\ot i} \ot (\one \ot M) )\ot M \cdots )\ot M \\
\xrightarrow[]{%
(\cdots(1_{M^{\ot i}}\ot(\eta \ot 1_M))\ot 1_M \cdots)\ot 1_M}
(\cdots(M^{\ot i} \ot (M \ot M) )\ot M \cdots )\ot M \\
\xrightarrow[]{%
(\cdots (\alpha_{M^{\ot i},M,M}^{-1}\ot 1_M \cdots )\ot 1_M}
M^{\ot(k+1)} \xrightarrow[]{f} X 
\end{cases}\qquad\text{for $0<i<k$}\\\\
\sigma^k(f) = 
\begin{cases}
M^{\ot k} = M^{\ot (k-1)}\ot M  
\xrightarrow[]{%
1_{M^{\ot(k-1)}}\ot \rho^{-1}_M} 
M^{\ot (k-1)}\ot (M\ot \one) \\
\xrightarrow[]{%
1_{M^{\ot (k-1)}}\ot (1_M \ot \eta)}
M^{\ot (k-1)}\ot (M\ot M)
\xrightarrow{%
\alpha^{-1}_{M^{\ot (k-1)}\ot (M\ot M)}}
M^{\ot(k+1)}
\xrightarrow[]{%
f}
X. 
\end{cases}
\end{cases}
\end{align*}
We refer to this construction, $(A,\delta,\sigma)$, as the {\em Hochschild cosimplicial object}. 
\end{defn}

In the definition of the codegeneracy maps for $0<i<k$ (in the definition of the Hochschild cosimplicial object above), we should observe that the inverse of the procedure to isolate a pair\footnote{This procedure to isolate a pair of objects in a tuple was described in Section~\ref{sec:tuphoch}. } is now used to rearrange the parenthesis back to $M^{\ot (k+1)}$. 

\begin{thm}\label{thm:coshoch}
The Hochschild cosimplicial object is, in fact, a cosimplicial object in $\Ab$. 
\end{thm}

\begin{proof}
This is a straightforward verification of the identities given in Proposition~\ref{prop:cosimp}. We will only show $\delta^1\delta^0=\delta^0\delta^0$ for
\begin{align*} 
\Hom_{\fK}(\one,X)\xrightarrow[]{}\Hom_{\fK}(M,X)\xrightarrow[]{}\Hom_{\fK}(M^{\ot 2},X). 
\end{align*}
We construct the following diagram, where $\delta^1\delta^0$
is on the left vertical (solid arrows) and $\delta^0\delta^0$
is on the right vertical (solid arrows). The dashed arrows are drawn to prove the claim. 
\begin{center}
\begin{tikzpicture}
\matrix(m)[matrix of math nodes,row sep=5em,column sep=5em,text height=1.5ex,text depth=0.25ex]
{ 
M    \ot M         &                      &              M \ot M   \\ 
M                  &                      &                        \\ 
\one \ot M         & \one \ot (M \ot M)   &   (\one \ot M) \ot M   \\ 
X    \ot M         & X    \ot (M \ot M)   &   (X    \ot M) \ot M   \\ 
                   &                      &    X           \ot M   \\ 
X                  &                      &    X                   \\ 
};
\path[ = ,font=\scriptsize](m-1-1) edge[-,double]         node[above]{$               $} (m-1-3);
\path[ = ,font=\scriptsize](m-6-1) edge[-,double]         node[above]{$               $} (m-6-3);
\draw[ ->,font=\scriptsize](m-1-1) edge         node[right]{$\mu $} (m-2-1);
\draw[ ->,font=\scriptsize](m-2-1) edge         node[right]{$\lambda^{-1}_M $} (m-3-1);
\draw[ ->,font=\scriptsize](m-3-1) edge         node[right]{$f \ot 1_M $} (m-4-1);
\draw[ ->,font=\scriptsize](m-4-1) edge         node[right]{$\omega $} (m-6-1);
\draw[ ->,font=\scriptsize](m-1-3) edge         node[right]{$\lambda^{-1}_M \ot 1_M $} (m-3-3);
\draw[ ->,font=\scriptsize](m-3-3) edge         node[right]{$(f \ot 1_M)\ot 1_M $} (m-4-3);
\draw[ ->,font=\scriptsize](m-4-3) edge         node[right]{$\omega \ot 1_M $} (m-5-3);
\draw[ ->,font=\scriptsize](m-5-3) edge         node[right]{$\omega $} (m-6-3);
\draw[ ->,font=\scriptsize](m-1-1) edge[dashed] node[right]{$\lambda^{-1}_{M\ot M} $} (m-3-2);
\draw[ ->,font=\scriptsize](m-1-3) edge[dashed] node[right]{$\lambda^{-1}_{M\ot M} $} (m-3-2);
\draw[ ->,font=\scriptsize](m-3-3) edge[dashed] node[above]{$\alpha_{\one,M,M} $} (m-3-2);
\draw[ ->,font=\scriptsize](m-3-2) edge[dashed] node[above]{$1_{\one}\ot\mu $} (m-3-1);
\draw[ ->,font=\scriptsize](m-3-2) edge[dashed] node[right]{$f \ot 1_{M \ot M} $} (m-4-2);
\draw[ ->,font=\scriptsize](m-4-3) edge[dashed] node[above]{$\alpha_{X,M,M} $} (m-4-2);
\draw[ ->,font=\scriptsize](m-4-2) edge[dashed] node[above]{$1_{X}\ot\mu $} (m-4-1);
%
\end{tikzpicture}
\end{center}
The triangular shaped square to the upper left commutes since $\lambda$ is a natural isomorphism. The middle triangle is the identity composed with $\lambda^{-1}_{M\ot M} $. The right upper triangle commutes by natural coherence in $\fK$ (see Remark~\ref{rem:coh}). The middle left square commutes by compositions with the identity, and we get $f \ot \mu$. The right middle square commutes since $\alpha$ is a natural isomorphism. The bottom part of the diagram commutes by using the right action property of $X$ ($X$ is a bimodule object of $M$). 

The other identities can be verified by a similar method by constructing the relevant diagram and use natural coherence or the other properties of monoid objects and bimodule objects discussed earlier in this paper.  
\end{proof}
%

We observe that the coface operations, $\delta^i$, in the Hochschild cosimplicial object $A$ in Definition~\ref{def:coshoch}, equal to the summands, $\chi_i$, we used to define the differentials, $d^i$, in the Hochschild cochain complex in Definition~\ref{def:hoch}. The cosimplicial object can then be turned into a cochain complex by adding $0$'s on the ``left side'' of the object, such that it is defined for negative $k$. This gives an equivalent formulation of Hochschild cohomology. 

\begin{thm}\label{def:comp1}
Let $(\fK,\ot,\one,\alpha,\lambda,\rho)$ be an $\Ab$-enriched monoidal category, $(M,\mu,\eta)$ a monoid object in $\fK$ and $(X,\nu,\omega)$ a bimodule object over $M$. Let $A$ be the Hochschild cosimplicial object defined in Definition~\ref{def:coshoch}. The following procedure turns the comsimplicial object $A$ into a cochain complex. For $k<0$ define $A^k=0$ and add these to the complex. The differential is defined to be the alternating sum of the coface operations for non-negative $k$ and $0$ otherwise, i.e.\ 
\begin{align*}
d^k=
\begin{cases}
0 &\text{if $k<0$}\\ 
\delta^0-\delta^1 &\text{if $k=0$} \\
\sum_{i=0}^{k+1}(-1)^i\delta^i&\text{if $k>0$}. 
\end{cases}
\end{align*}
Moreover, this cochain complex equals the Hochschild cochain complex defined in Definition~\ref{def:hoch}. 
\end{thm}
%
%
%
%
%
%
%
%
%
%
%
%
%
%
%
%
%
%
%
%
%
%
\mbox{}\\
\noindent\textbf{Acknowledgements.} I would like to thank the organisers of ISCRA, Professor Javad Asadollahi and his crew for a wonderful convention in Isfahan, Iran, April 2019. It was a true pleasure and I am grateful to be given the opportunity to speak at the conference. I thank Professor Aslak Bakke Buan and my former employer the Department of Mathematical Sciences at NTNU for providing funding in this occasion. I am further thankful for the invitation to submit this work to the Bulletin of the Iranian Mathematical Society. I thank the two anonymous reviewers for valuable comments and suggestions, and Professor Majid Soleimani-damaneh, Editor in Chief, Bulletin of the Iranian Mathematical Society, for the correspondence. Finally, I thank Eirik Hellstrøm Finnsen for his time and competence concerning this article. 

%
%
%
%
%
%
%
%
%
\bibliography{bibFile}{} 
\bibliographystyle{ieeetr}

\end{document}